\documentclass[11pt]{article}

\usepackage{amsmath,amsfonts}
\usepackage[initials]{amsrefs}
\usepackage{amsthm}
\usepackage[all]{xy}
\usepackage{hyperref}

\numberwithin{equation}{section}

\newcommand{\C}{\mathbb{C}}

\newcommand{\Z}{\mathbb{Z}}
\newcommand{\Ptau}{P^{E_\tau}}

\providecommand{\abs}[1]{\left\lvert#1\right\rvert}

\newcommand{\bbracket}[1]{\left[#1\right]}
\newcommand{\fbracket}[1]{\left\{#1\right\}}
\newcommand{\bracket}[1]{\left(#1\right)}

\newcommand{\dpa}[1]{{\pa\over \pa #1}}
\newcommand{\vjump}[1]{\vskip #1\baselineskip}

\newcommand{\bl}{\textbf}

\newcommand{\mc}{\mathcal}
\newcommand{\cinfty}{C^{\infty}}
\newcommand{\pa}{\partial}

\newcommand{\iso}{\cong}
\newcommand{\suml}{\sum\limits}

\DeclareMathOperator{\im}{Im}

\theoremstyle{plain}
\newtheorem{thm}{Theorem}[section]
\newtheorem{thm-defn}{Theorem/Definition}[section]
\newtheorem{lem}{Lemma}[section]
\newtheorem{prop}{Proposition}[section]
\newtheorem{cor}{Corollary}[section]

\newtheorem*{claim}{Claim}

\theoremstyle{definition}
\newtheorem{dfn}{Definition}[section]

\newtheorem{eg}{Example}[section]

\theoremstyle{remark}

\allowdisplaybreaks[4]  

\begin{document}

 \title{\mbox{Feynman Graph Integrals and Almost Modular Forms}}
  \author{Si Li}
  \date{}

  \maketitle


\begin{abstract} We introduce a type of graph integrals on elliptic curves from the heat kernel. We show that such graph integrals have modular properties under the modular group $SL(2, \Z)$, and prove the polynomial nature of the anti-holomorphic dependence.
\end{abstract}



\section{Introduction}
Modular forms arise naturally in physics as correlation functions of quantum system with modular groups as symmetries. One such example is the topological string theory on Calabi-Yau manifolds. The topological string produces geometric invariants $F_g$ for each non-negative integer $g$, which can be viewed as generalized modular forms on the Calabi-Yau moduli space. However, $F_g$'s are in general not holomorphic objects.  They satisfy the holomorphic anomaly equation as shown in the work of Bershadsky, Cecotti, Ooguri and Vafa \cite{BCOV}. When the Calabi-Yau is an elliptic curve, R. Dijkgraaf \cite{Dijkgraaf-mirror} anticipated the interpretation of $F_g$ in the language of almost modular forms. Later, Aganagic, Bouchard and Klemm  \cite{topstring-modularform} generalize almost modular forms to describe local Calabi-Yau models in topological strings.

In this paper, we will focus on elliptic curves and their moduli. By ``almost modular form" of weight $k$, we mean in a weak sense of \cite{almost-modular-form}: a function $f(\tau, \bar \tau)$ on the upper half-plane $\mathcal H$, which is modular of weight $k$
\begin{align*}
          f(\gamma \tau, \overline{\gamma\tau})=(C\tau+D)^k f(\tau, \bar\tau), \ \ \forall \gamma=\begin{pmatrix} A& B\\ C & D \end{pmatrix}\in SL(2,\Z)
\end{align*}
where $\gamma \tau={A\tau+B\over C\tau +D}$, and the anti-holomorphic dependence of $f$ is of polynomial in ${1\over \im \tau}$, i.e.
\begin{align*}
         f(\tau, \bar\tau)=\sum_{i=0}^N f_i(\tau){1\over \bracket{\im \tau}^i}
\end{align*}
for some non-negative integer $N$ and holomorphic functions $f_i(\tau)$. The famous $\bar\tau\to \infty$ limit \cite{BCOV} picks up the leading holomorphic term $f_0(\tau)$ in this context, which is quasi-modular \cite{almost-modular-form}.

Motivated by  topological string on elliptic curves \cite{Dijkgraaf-mirror, Si-Kevin, thesis}, we consider the following integral $W_{\Gamma}$ associated with any graph $\Gamma$: for each edge of $\Gamma$, we associate a kernel function constructed from the heat kernel; for each vertex, we associate a copy of integration on the elliptic curve. See section \ref{section-graph-integral} for precise definitions. $W_{\Gamma}$ depends on the complex structure of the elliptic curve, and can be viewed as a function on the upper half-plane.

\begin{thm}
$W_\Gamma$ is an almost modular form of weight $2|E(\Gamma)|$ in the above sense. Here $|E(\Gamma)|$ is the number of edges in $\Gamma$.
\end{thm}

We can also put certain holomorphic derivatives on the propagator and obtain the graph integral for a decorated graph. The resulting graph integral is again almost modular form with specific weight. See Corollary \ref{almost-modular-result}. As shown in \cite{Dijkgraaf-mirror, Si-Kevin, thesis}, $F_g$ on elliptic curves for each $g$ is given by combinatorial sum of graph integrals of this type.

The paper is organized as follows: In section 2, we describe the BCOV propagator which is the building block of the graph integral. In section 3, we define the graph integral considered in this paper. In section 4, we prove the modular property of graph integrals. In section 5, we prove that the graph integral has polynomial dependence in ${1\over \im \tau}$. In the appendix, we provide all the technical details of the necessary estimates for the graph integrals. \\

\noindent\bl{Acknowledgement}: The author thanks K. Costello for many stimulating discussions on two dimensional quantum field theory, and thanks S.T. Yau for useful conversations on quasi-modular forms.

\section{BCOV propagator on the elliptic curve}
Let $\mathcal H=\fbracket{\tau\in \C| \im \tau>0}$ be the complex upper half-plane. Let $E_\tau=\C/\Lambda_\tau$ be the elliptic curve associated with the lattice
$$
  \Lambda_\tau={\Z\oplus \tau \Z}, \quad \tau\in \mathcal H
$$
We will use $z$ for the standard linear coordinate on $\C$, such that $E_\tau$ is obtained via the equivalence $z\sim z+1\sim z+\tau$. The notation $d^2z$ will always refer to the following measure on $\C$ or $E_\tau$
$$
   d^2z= {i\over 2}dz\wedge d\bar z
$$
Let
$$
    \Delta=-4\dpa{z}\dpa{\bar z}
$$
be the standard flat Laplacian operator on $E_\tau$. We consider the kernel function $K_t^{E_\tau}$ for the heat operator $e^{-t\Delta}$
\begin{align}
    K_t^{E_\tau}(z_1,\bar z_1;,z_2, \bar z_2)={1\over 4\pi t}\sum_{\lambda\in \Lambda_\tau}e^{-\abs{z_1-z_2+\lambda}^2/4t}, \ \ t>0
\end{align}
which is the unique function solving the heat equation
$$
          \bracket{\dpa{t}+\Delta_{z_1}}K_t^{E_\tau}(z_1, \bar z_1;,z_2, \bar z_2)=0
$$
and the initial condition
$$
         \lim_{t\to 0} \int_{E_\tau} d^2z_2 K_t^{E_\tau}(z_1, \bar z_1; z_2, \bar z_2)\phi(z_2, \bar z_2)=\phi(z_1, \bar z_1), \quad \forall \phi\in \cinfty(E_\tau)
$$

\begin{dfn}The \bl{BCOV propagator} $\Ptau_{\epsilon,L}$ is defined to be the smooth kernel function
\begin{align}
     P_{\epsilon,L}^{E_\tau}(z_1, \bar z_1;z_2, \bar z_2)=\int_\epsilon^L{dt}\bracket{\dpa{z_1}}^2 K_t^{E_\tau}(z_1, \bar z_1;,z_2, \bar z_2), \quad \epsilon, L>0
\end{align}
representing the operator $\int_\epsilon^L {dt}\bracket{\dpa{z}}^2 e^{-t\Delta_z}$. We will also use $\Ptau_{0,\infty}(z_1,z_2)$ to represent the limit
\begin{align}
     \Ptau_{0,\infty}(z_1,z_2)\equiv \lim_{\substack{\epsilon\to 0\\ L\to \infty }}\Ptau_{\epsilon, L}(z_1, \bar z_1;,z_2, \bar z_2)
\end{align}
which is singular at $z_1=z_2$.
\end{dfn}

Note that we have dropped the anti-holomorphic dependence in $\Ptau_{0,\infty}$. It's shown in the next lemma that it's holomorphic away from the diagonal $z_1=z_2$.  The kernel $\Ptau_{\epsilon, L}$ is motivated from string theory. It describes the propagator of the Kodaira-Spencer gauge theory, which is originally introduced in \cite{BCOV} on Calabi-Yau 3-folds, and generalized in \cite{Si-Kevin, thesis} to Calabi-Yau manifolds of arbitrary dimensions.

The following Lemma \ref{bcov propagator} and Lemma \ref{self-loop} for Green functions on elliptic curves  are well-known:
\begin{lem}\label{bcov propagator}
\begin{align}
\begin{split}
       \Ptau_{0,\infty}(z_1,z_2)={1\over 4\pi}\wp\bracket{z_1-z_2;\tau}+{\pi\over 12}E_2^*(\tau,\bar\tau)\quad \forall z_1\neq z_2
\end{split}
\end{align}
Here $\wp$ is the Weierstrass's elliptic function
$$
   \wp(z;\tau)={1\over z^2}+\sum_{\lambda\in \Lambda_\tau-\{0\}}\bracket{{1\over (z-\lambda)^2}-{1\over \lambda^2}}
$$
$ E_2^*(\tau,\bar\tau)=E_2(\tau)-{3\over \pi \im \tau}$, and
$$
           E_2(\tau)=1-24\sum_{n=1}^\infty {nq^n\over 1-q^n}, \quad q=e^{2\pi i \tau}
$$
is the second Eisenstein series.
\end{lem}
$\Ptau_{0,\infty}(z_1,z_2)$ becomes singular as $z_1$ approaches $z_2$, due to the singularity from the Green kernel. However, if we change the order of the limit, we have
\begin{lem}\label{self-loop}
\begin{align}
   \lim_{\substack{\epsilon\to 0\\ L\to \infty}}\lim_{z_1\to z_2}\pa_{z_1}^n\Ptau_{\epsilon,L}(z_1,z_2)=\begin{cases}{1\over 12\pi}E_2^* & \text{if $n=0$} \\{(n+1)!\zeta(n+2)\over 2\pi}E_{n+2} & \text{if $n>0$ is even}\\ 0 & \text{if $n$ is odd} \end{cases}
\end{align}
where $E_{2k}$ is the Eisenstein series of weight $2k$.
\end{lem}

An elementary proof of Lemma \ref{bcov propagator} and Lemma \ref{self-loop} is given in Appendix \ref{lemma-section-2}.

The objects $\lim\limits_{\substack{\epsilon\to 0\\ L\to \infty}}\lim\limits_{z_1\to z_2}\pa_{z_1}^n\Ptau_{\epsilon,L}(z_1,z_2)$ are special examples of the Feynman graph integrals to be discussed in the next section. They correspond to self-loops, and have nice modular properties. In fact, they are examples of \emph{almost holomorphic modular forms} \cite{almost-modular-form}. $E_2^*$ plays a special role, which is modular but not holomorphic in $\tau$. However, its anti-holomorphic dependence is very mild, i.e. polynomial in ${1\over \im \tau}$. We will see that a large class of graph integrals will also have this property.

\section{Feynman graph integral}\label{section-graph-integral}
We consider a directed graph $\Gamma$. Let $V(\Gamma)$ be the set of vertices, $E(\Gamma)$ be the set of edges, and
$$
 t, h: E\to V
$$
be the assignments of tail and head to each directed edge. We will also consider the decorated graph
$$
   \bracket{\Gamma, n}\equiv (\Gamma, \fbracket{n_e}_{e\in E})
$$
where the decoration is given by
$$
    n: E(\Gamma)\to \Z^{\geq 0}, \ \ e\to n_e
$$
which associates each edge a non-negative integer. In the case that $n$ is the zero map, we will simply ignore $n$ and write $\Gamma$ for $(\Gamma,n)$.

Given a decorated graph $(\Gamma,n)$ and elliptic curve $E_\tau$, we associate the following graph integral
\begin{align}
              W_{\bracket{\Gamma, n}}\bracket{\Ptau_{\epsilon,L}}
              =\prod_{v\in V(\Gamma)}\int_{E_\tau}{d^2z_v\over \im \tau} \prod_{e\in E(\Gamma)}
               \bracket{\dpa{z_{h(e)}}}^{n_e}\Ptau_{\epsilon,L;e}
\end{align}
where $\Ptau_{\epsilon,L;e}=\Ptau_{\epsilon,L}(z_{h(e)}, \bar z_{h(e)};z_{t(e)},\bar z_{t(e)})$. The propagator $\Ptau_{\epsilon, L}$ is smooth as long as $\epsilon, L>0$, but exhibits singularity at the diagonal as $\epsilon\to 0$. However, the  graph integral $W_{\bracket{\Gamma, n}}\bracket{\Ptau_{\epsilon,L}}$ has better behavior.

\begin{lem}
The following limit exists
\begin{align*}
   \lim_{\substack{\epsilon\to 0 \\ L\to \infty}}W_{\bracket{\Gamma, n}}\bracket{\Ptau_{\epsilon,L}}
\end{align*}
\end{lem}
\begin{proof} By Lemma \ref{self-loop}, we can assume that $\Gamma$ is connected and has no self-loops. The singularity of $ W_{\bracket{\Gamma, n}}\bracket{\Ptau_{\epsilon,L}}$ comes from the diagonals of the propagator as $\epsilon\to 0$. Let's fix $L$ first and analyze the limit $\epsilon\to 0$.

Let's first fix some notations. In the graph integral, we have associated a copy of $E_\tau$ for each $v\in V(\Gamma)$, which we will distinguish by $E_{v}$. Let $d$ be the distance function on $E_\tau$ with respect to the flat metric. Let $\chi: [0,\infty)\to [0,1]$ be a smooth function with $\chi(x)=1$ if $x<\delta$ and $\chi(x)=0$ if $x>2\delta$, where $\delta\ll 0$ is a sufficient small positive number.  Define
$$
        K_t^{\delta}(z_1,\bar z_1;z_2,\bar z_2)=\chi(d(z_1,z_2)^2){1\over 4t}e^{-d(z_1,z_2)^2/4t}, \ \ \forall z_1,z_2\in E_\tau
$$
and
$$
          \hat K_t^{E_\tau}=K_t^{E_\tau}-K_t^{\delta}
$$
Then $\hat K_t^{E_\tau}$ is smooth as $t\to 0$. Similarly we define
$$
        P^\delta_{\epsilon, L}(z_1,\bar z_1;z_2,\bar z_2)=\int_\epsilon^L dt\ \pa_{z_1}^2K_t^\delta(z_1,\bar z_1;z_2,\bar z_2), \ \ \hat P^\delta_{\epsilon, L}=\Ptau_{\epsilon, L}-P^\delta_{\epsilon, L}
$$
$\hat P^\delta_{\epsilon, L}$ is smooth as $\epsilon\to 0$ and $P^\delta_{\epsilon, L}$ contains all the information about the singularity.

The graph integral $W_{\bracket{\Gamma, n}}\bracket{\Ptau_{\epsilon,L}}=W_{\bracket{\Gamma, n}}\bracket{P^\delta_{\epsilon,L}+\hat P^\delta_{\epsilon, L}}$ splits into a sum of graph integrals where we associate $P^\delta_{\epsilon, L}$ or $\hat P^\delta_{\epsilon, L}$ on each edge. Let's pick up a particular term, and let $\Gamma^\prime$ be the corresponding subgraph of $\Gamma$ consisting of those edges assigned with the singular propagator $P^\delta_{\epsilon, L}$. Let $\Gamma^\prime=\Gamma_1\cup\cdots\cup \Gamma_k$ be the decomposition into connected components. It's sufficient to show that each connected component $\Gamma_i$ contributes a regular integral as $\epsilon\to 0$.

Let's focus on one component $\Gamma_1$. Let $v_\bullet\in V(\Gamma_1)$ be an arbitrary vertex. The integral is supported near the diagonal of $\prod\limits_{v\in V(\Gamma_1)} E_v$, which can be identified with a small neighborhood of zero section of the vector bundle $T_{E_{v_\bullet}}^{\oplus (|V(\Gamma_1)|-1)}\iso E_{v_\bullet}\times \C^{\oplus (|V(\Gamma_1)|-1)}$ on $E_{v_\bullet}$. Here $T_{E_{v_\bullet}}$ is the tangent bundle of $E_{v_\bullet}$. Therefore we can write the relevant graph integral on $\Gamma_1$ into the form
$$
  \int_{E_{v_\bullet}}{d^2z_{v_\bullet}} \prod_{v\in V(\Gamma_1)\backslash \{v_\bullet\}}\int_{\C} d^2 y_v \bracket{\prod_{e\in E(\Gamma_1)}\pa_{y_e}^{m_e}H_\epsilon^L(y_e, \bar y_e)}\Phi
$$
where $H_\epsilon^L(z,\bar z)=\int_\epsilon^L {dt\over 4\pi t}e^{-|z|^2}/4t$, $m_e$ some non-negative integers for each edge $e\in E(\Gamma_1)$,
$$
  y_e=\begin{cases} y_{h(e)} & \text{if}\ t(e)=v_\bullet\\ -y_{t(e)} & \text{if}\ h(e)=v_\bullet \\ y_{h(e)}-y_{t(e)}& \text{otherwise} \end{cases}
$$
and $\Phi$ is a smooth function on $E_{v_\bullet}\times \C^{\oplus (|V(\Gamma_1)|-1)}\times \prod\limits_{v\notin V(\Gamma_1)}E_v$ with compact support. By Proposition \ref{finiteness lem} of Appendix \ref{appendix-finiteness-lemma} and its proof, the above integral is regular and uniformly convergent as $\epsilon\to 0$. This proves that $\lim\limits_{\epsilon\to 0}W_{\bracket{\Gamma, n}}\bracket{\Ptau_{\epsilon,L}}$ exists.

Now we consider the limit $L\to \infty$. Since $\Ptau_{\epsilon, \infty}=\Ptau_{\epsilon, L}+\Ptau_{L, \infty}$ and the kernel function $\Ptau_{L, \infty}$ is smooth. It follows that
$$
\lim_{\epsilon\to 0}W_{\bracket{\Gamma, n}}\bracket{\Ptau_{\epsilon,\infty}}=\lim_{\epsilon\to 0}W_{\bracket{\Gamma, n}}\bracket{\Ptau_{\epsilon,L}+\Ptau_{L,\infty}}
$$
exits. This proves the lemma.
\end{proof}

\begin{dfn} Given a decorated graph $(\Gamma, n)$, we define a smooth function $W_{(\Gamma,n)}$ on $\mathcal H$ by
\begin{align}
     W_{\bracket{\Gamma,n}}(\tau,\bar\tau)\equiv \lim_{\substack{\epsilon\to 0 \\ L\to \infty}}W_{\bracket{\Gamma, n}}\bracket{\Ptau_{\epsilon,L}}
\end{align}
\end{dfn}

\begin{eg} Consider the self-loop graph with decoration $n$.
$$
  {{\xymatrix{\bullet\ar@(ul,ur)[]|{n}}}}
$$
Lemma \ref{self-loop} implies that
\begin{align*}
   W_{{\xymatrix{\bullet\ar@(ul,ur)[]|{n}}}}=\begin{cases}{1\over 12\pi}E_2^* & \text{if $n=0$} \\{(n+1)!\zeta(n+2)\over 2\pi}E_{n+2} & \text{if $n>0$ is even}\\ 0 & \text{if $n$ is odd} \end{cases}
\end{align*}

\end{eg}

\section{Modularity}
We consider the modular group $SL(2.\Z)$, which acts on $\mc H$ by
$$
   \tau\to \gamma\tau={A\tau+B\over C\tau+D}, \ \ \mbox{for}\ \gamma\in \begin{pmatrix} A& B\\ C& D \end{pmatrix} \in SL(2, \Z)
$$
Recall that a function $f: \mc H\to \C$ is said to have weight $k$ under the modular group $SL(2,\Z)$ if
$$
          f(\gamma \tau)=(C\tau+D)^k f(\tau) \ \ \mbox{for all}\ \gamma\in \begin{pmatrix} A& B\\ C& D \end{pmatrix} \in SL(2, \Z)
$$

\begin{prop}
The graph integral $W_{\bracket{\Gamma,n}}(\tau,\bar\tau)$ has weight $\sum\limits_{e\in E(\Gamma)}(n_e+2)$ under $SL(2, \Z)$.
\end{prop}
\begin{proof}Given $\gamma\in \begin{pmatrix} A& B\\ C& D \end{pmatrix} \in SL(2, \Z)$, the lattice transforms as $\Lambda_{\gamma \tau}={1\over (C\tau+D)}\Lambda_\tau$. It follows that the propagator has the transformation property
\begin{align}
\begin{split}
   &\pa^{m}_{z_1}P^{E_{\gamma\tau}}_{\epsilon, L}(z_1,\bar z_1;,z_2,\bar z_2)\\=&(C\tau+D)^{m+2}\left(\pa^{m}_{z_1}\Ptau_{|C\tau+D|^2\epsilon,|C\tau+D|^2L}\right)\left(
   z_1^\prime, \bar z_1^\prime; z_2^\prime, \bar z_2^\prime\right)
\end{split}
\end{align}
where $z_i^\prime=(C\tau+D)z_i, i=1,2$. Using the modular invariance of the measure ${d^2z\over \im \tau}$, we find
\begin{align*}
   W_{\bracket{\Gamma, n}}\bracket{P^{E_{\gamma\tau}}_{\epsilon,L}}&=\prod_{v\in V(\Gamma)}\int_{E_{\gamma\tau}}{d^2z_v\over \im (\gamma\tau)} \prod_{e\in E(\Gamma)}
               \bracket{\pa_{z_{h(e)}}^{n_e}P^{E_{\gamma\tau}}_{\epsilon,L}}(z_{h(e)},\bar z_{h(e)};z_{t(e)},\bar z_{t(e)})\\
               &=\prod_{v\in V(\Gamma)}\int_{E_\tau}{d^2z_v\over \im \tau} \prod_{e\in E(\Gamma)} (C\tau+D)^{n_e+2}
              \pa_{z_{h(e)}}^{n_e}\Ptau_{|C\tau+D|^2\epsilon,|C\tau+D|^2 L}(z_{h(e)},\bar z_{h(e)};z_{t(e)},\bar z_{t(e)})\\
               &=\bracket{C\tau+D}^{\sum\limits_{e\in E(\Gamma)}(n_e+2)}W_{\bracket{\Gamma, n}}\bracket{P^{E_{\tau}}_{|C\tau+D|^2\epsilon,|C\tau+D|^2L}}
\end{align*}
The proposition follows after taking the limit $\epsilon\to 0, L\to \infty$.
\end{proof}

\section{Anti-holomorphic dependence}
The function $W_{(\Gamma, n)}$ has particular weight under modular transformation. However,  it's not  holomorphic in general. We have seen this when $\Gamma$ is a one-vertex graph with a self-loop. In this case $W_{\Gamma}={1\over 12\pi}E_2^*$, which exhibits a polynomial dependence on ${1\over \im \tau}$. In this section, we will show that the $\bar\tau$ dependence of any graph integral is polynomial in ${1\over \im\tau}$.

\begin{prop}\label{antiholomorphic-dependence}
For any decorated graph $(\Gamma, n)$, the graph integral can be decomposed as
$$
       W_{(\Gamma, n)}(\tau,\bar\tau)=\sum_{i=0}^N f_i(\tau){1\over \bracket{\im \tau}^i}
$$
where $f_i(\tau)$'s are holomorphic functions on $\mathcal H$, and $N$ is some non-negative integer.
\end{prop}
\begin{proof} We will show that $\pa_{\bar\tau}W_{(\Gamma, n)}(\tau,\bar\tau)$ is also some graph integral with fewer edges. The proposition will then follow by induction. Without loss of generality, we can assume that $\Gamma$ has no self-loops.

First of all, it's easy to see that
$$
   \dpa{\bar\tau}\int_{E_\tau}{d^2z_\tau\over \im \tau} f(z,\bar z;\tau, \bar \tau)=\int_{E_\tau}{d^2z_\tau\over \im \tau}
   \bracket{{\im z\over \im \tau}\dpa{\bar z}+\dpa{\bar\tau}}f(z,\bar z;\tau, \bar \tau)
$$
Here the integration on $E_\tau$ is performed in the region $\{a+b\tau| 0\leq a,b\leq 1\}$. Hence
\begin{align*}
              \pa_{\bar\tau}W_{\bracket{\Gamma, n}}\bracket{\Ptau_{\epsilon,L}}
              &=\prod_{v\in V(\Gamma)}\int_{E_\tau}{d^2z_v\over \im \tau}
              \sum_{e\in E(\Gamma)}\bbracket{\bracket{{\im \bracket{z_{h(e)}-z_{t(e)}}\over \im \tau}\dpa{\bar z_{h(e)}}+\dpa{\bar\tau}}
             \pa_{z_{h(e)}}^{n_e}\Ptau_{\epsilon, L,e}}\\
             &\quad\qquad\bracket{\prod_{e^\prime\in E(\Gamma)\backslash \{e\}}
               \pa_{z_{h(e^\prime)}}^{n_{e^\prime}}\Ptau_{\epsilon,L,e^\prime)}}
\end{align*}
To simplify the notation, we will write
$$
z_e\equiv z_{h(e)}-z_{t(e)}
$$
for any $e\in E(\Gamma)$. Using the heat equation, we find
\begin{align*}
\bracket{{\im \bracket{z_{h(e)}-z_{t(e)}}\over \im \tau}\dpa{\bar z_{h(e)}}+\dpa{\bar\tau}}
             \pa_{z_{h(e)}}^{n_e}\Ptau_{\epsilon, L,e}\\
  =\left.\sum_{\lambda\in \Lambda_\tau}{\im \bracket{z_e-\lambda}\over \im \tau} {1\over 16\pi t} \pa_{z_e}^{n_e+1}e^{-|z_e-\lambda|^2/4t}\right|_{t=\epsilon}^{t=L}
\end{align*}
$ \pa_{\bar\tau}W_{\bracket{\Gamma, n}}\bracket{\Ptau_{\epsilon,L}}$ has two types of contributions corresponding to $t=\epsilon$ or $t=L$ in the above formula.

\subsubsection*{The term with $t=L$} Let's first consider the term with $t=L$. If $n_e>0$, then the summation $\sum\limits_{\lambda\in \Lambda_\tau}{\im \bracket{z_e-\lambda}\over \im \tau} {1\over 16\pi t} \pa_{z_e}^{n_e+1}e^{-|z_e-\lambda|^2/4t}$ is absolutely convergent and uniform in $t$, so
$$
\lim_{L\to \infty}\sum\limits_{\lambda\in \Lambda_\tau}{\im \bracket{z_e-\lambda}\over \im \tau} {1\over 16\pi L} \pa_{z_e}^{n_e+1}e^{-|z_e-\lambda|^2/4L}=0
$$

If $n_e=0$, then
\begin{align*}
        &\sum\limits_{\lambda\in \Lambda_\tau}{\im \bracket{z_e-\lambda}\over \im \tau} {1\over 16\pi L} \pa_{z_e}e^{-|z_e-\lambda|^2/4L}\\
        =&\sum_{n\in \Z}\bracket{{\im z_e\over \im \tau}-n}\sum_{m\in \Z}\bracket{{\bar z_e-(m+n\bar \tau)\over 64\pi L^2}}e^{-\abs{z_e-(m+n\tau)}^2/4L}\\
        =&\sum_{n\in \Z}\bracket{{\im z_e\over \im \tau}-n}\sum_{m\in \Z}\left[\bracket{{\bar z_e-(m+n\bar \tau)\over 64\pi L^2}}e^{-\abs{z_e-(m+n\tau)}^2/4L}\right.\\
        & \quad \left.-\int_{m}^{m+1}dy \bracket{{\bar z_e-(y+n\bar\tau)\over 64\pi L^2}}e^{-\abs{z_e-(y+n\tau)}^2/4L}\right]\\
        &+ \sum_{n\in \Z}\bracket{{\im z_e\over \im \tau}-n}\int_{-\infty}^{\infty}dy \bracket{{\bar z_e-(y+n\bar\tau)\over 64\pi L^2}}e^{-\abs{z_e-(y+n\tau)}^2/4L}\\
        &=I_1+I_2
\end{align*}
Similarly we have $\lim\limits_{L\to \infty}I_1=0$. $I_2$ can be computed using Gaussian integral
\begin{align*}
       I_2=&\sum_{n\in \Z}\bracket{{\im z_e\over \im \tau}-n} \bracket{ \im z_e-n\im \tau\over 32 i \sqrt{\pi}L^{3/2}}e^{-\bracket{\im z_e-n\im \tau}^2/4L}\\
       =&{1\over 32 i \sqrt{\pi}\bracket{\im \tau}^2} \sum_{n\in \Z} {\bracket{{\im z_e\over \im \tau}-n}^2\over \tilde L^{3/2}}e^{-\bracket{{\im z_e\over \im \tau}-n}^2/4\tilde L}\qquad \text{where}\ \tilde L={L\over \bracket{\im \tau}^2}\\
       =&{1\over 8 i \bracket{\im \tau}^2}\sum_{m\in \Z}\bracket{1-8\tilde L \pi^2 m^2}e^{-4m^2\pi^2 \tilde L+2\pi i m {\im z_e\over \im \tau}}
\end{align*}
where in the last step we have used Fourier transformation. Therefore
$$
  \lim_{L\to\infty}I_2={1\over 8 i \bracket{\im \tau}^2}
$$
To summarize, we find
$$
\lim_{L\to \infty}\sum\limits_{\lambda\in \Lambda_\tau}{\im \bracket{z_e-\lambda}\over \im \tau} {1\over 16\pi L} \pa_{z_e}^{n_e+1}e^{-|z_e-\lambda|^2/4L}=\begin{cases} {1\over 8 i \bracket{\im \tau}^2} & \text{if}\ n_e=0\\
      0 & \text{if}\ n_e>0 \end{cases}
$$

\subsubsection*{The term with $t=\epsilon$} Now we consider the term with $t=\epsilon$. Its contribution to $ \pa_{\bar\tau}W_{\bracket{\Gamma, n}}$ is
\begin{align*}
&\prod_{v\in V(\Gamma)}\int_{E_\tau}{d^2z_v\over \im \tau}
              \sum_{e\in E(\Gamma)}\bracket{\sum_{\lambda\in \Lambda_\tau}{\im \bracket{z_e-\lambda}\over \im \tau} {1\over 16\pi \epsilon} \pa_{z_e}^{n_e+1}e^{-|z_e-\lambda|^2/4\epsilon}}\bracket{\prod_{e^\prime\in E(\Gamma)\backslash \{e\}}
               \pa_{z_{h(e^\prime)}}^{n_{e^\prime}}\Ptau_{\epsilon,L;e^\prime}}\\
 =&\sum_{e\in E(\Gamma)}\bracket{ \prod_{v\in V(\Gamma)\backslash \{h(e)\}}\int_{E_\tau}{d^2z_v\over \im \tau}}\int_{\C} {d^2 z_{h(e)}\over \bracket{\im \tau}^2}\bracket{{\im z_e\over 16\pi \epsilon}\pa_{z_{e}}^{n_e+1}e^{-|z_e|^2/4\epsilon}}\bracket{\prod_{e^\prime\in E(\Gamma)\backslash \{e\}}
               \pa_{z_{h(e^\prime)}}^{n_{e^\prime}}\Ptau_{\epsilon,L;e^\prime}}
\end{align*}
By Proposition \ref{deformed-bracket}, it reduces to certain graph integral on $\Gamma^\prime$ under the limit $\epsilon\to 0, L\to \infty$, with an extra factor proportional to ${1\over \bracket{\im \tau}^2}$. Here $\Gamma^\prime$ is obtained from $\Gamma$ by collapsing one edge.

Combining the terms for $t=L$ and terms for $t=\epsilon$, it follows by induction that
$$
    \pa_{\bar\tau}W_{(\Gamma, n)}= \lim_{\substack{\epsilon\to 0\\ L\to \infty}} \pa_{\bar\tau}W_{\bracket{\Gamma, n}}\bracket{\Ptau_{\epsilon,L}}
    ={1\over \bracket{\im \tau}^2}\sum_{i=0}^K f_i(\tau){1\over \bracket{\im \tau}^i}
$$
for some holomorphic function $f_i(\tau)$ and non-negative integer $K$. Therefore $W_{(\Gamma, n)}$ has polynomial dependence on ${1\over \im \tau}$ as well.
\end{proof}

\begin{cor}
Let $\Gamma$ be a graph such that every two vertices are connected by at most one edge, then
\begin{align}
     {\pa_{\bar\tau}}W_{\Gamma}={i\over 8\bracket{\im\tau}^2}\sum_{e\in E(\Gamma)}\bracket{W_{\Gamma/e}-W_{\Gamma\backslash e}}
\end{align}
where $\Gamma/e$ is the graph by collapsing the edge $e$ in $\Gamma$, and $\Gamma\backslash e$ is the graph by deleting the edge $e$ in $\Gamma$.
\end{cor}
\begin{proof} In the proof of Proposition \ref{antiholomorphic-dependence}, there're two contributions. The term with $L\to\infty$ contributes ${1\over 8 i(\im \tau)^2} W_{\Gamma\backslash e}$. The term with $\epsilon\to 0$ comes from the integration
$$
\int_{\C}{d^2 z_{h(e)}\over \bracket{\im \tau}^2}\bracket{{\im z_e\over 16\pi \epsilon}\pa_{z_{e}}e^{-|z_e|^2/4\epsilon}}={i\over 8 (\im \tau)^2}\int_{\C}d^2 z_{h(e)}{1\over 4\pi \epsilon}e^{-|z_e|^2/4\epsilon}
$$
which becomes ${i\over 8 (\im \tau)^2} \delta_{z_e,0}$ as $\epsilon\to 0$.

\end{proof}

\begin{cor}\label{almost-modular-result}
For any decorated graph $(\Gamma, n)$, the graph integral $W_{(\Gamma,n)}(\tau, \bar\tau)$ is an almost  modular form of weight $\sum\limits_{e\in E(\Gamma)}(n_e+2)$.
\end{cor}
  \appendix
\section{The BCOV propagator}\label{lemma-section-2}
In this section, we give an elementary proof of Lemma \ref{bcov propagator} and Lemma \ref{self-loop}.
\begin{proof}[Proof of Lemma \ref{bcov propagator}] Let $z_{12}=z_1-z_2$.
\begin{align*}
   \Ptau_{\epsilon,L}(z_1,\bar z_1;z_2,\bar z_2)&=\int_\epsilon^L \sum_{m,n\in \Z}\bracket{{\bar z_{12}-\bracket{m+n\bar \tau}\over 4t}}^2 \exp\bracket{-\abs{z_{12}-\bracket{m+n\tau}}^2/4t}\\
   &=\int_\epsilon^L \sum_{m\in \Z}\bracket{{\bar z_{12}-m\over 4t}}^2\exp\bracket{-\abs{z_{12}-m}^2/4t}\\\\
&+\int_\epsilon^L {dt\over 4\pi t} \sum_{n\neq 0}\sum_{m\in \Z}\left[ \left({\bar z_{12}-(m+n\bar \tau)\over 4t}\right)^2 \exp\left(-|z_{12}-(m+n\tau)|^2/4t  \right)\right.\\
&\quad \left.-\int_m^{m+1}dy
\left({\bar z_{12}-(y+n\bar \tau)\over 4t}\right)^2 \exp\left(-|z_{12}-(y+n\tau)|^2/4t  \right)    \right]\\
&+\int_\epsilon^L {dt\over 4\pi t} \sum_{n\neq 0}\int_{-\infty}^\infty dy\left({\bar z_{12}-(y+n\bar \tau)\over 4t}\right)^2 \exp\left(-|z_{12}-(y+n\tau)|^2/4t  \right)  \\
&= I_1+ I_2+I_3
\end{align*}

$I_1$ is absolutely convergent and
\begin{align*}
    \lim_{\substack{\epsilon\to 0\\ L\to \infty}}I_1&=\int_0^\infty {dt\over 4\pi t} \sum_{m\in \Z} \left({\bar z_{12}-m\over 4t}\right)^2 \exp\left(-|z_{12}-m|^2/4t  \right)\\
&=\sum_{m\in \Z} {1\over (z_{12}-m)^2}\int_0^\infty  {dt\over 4\pi t}{1\over (4t)^2} \exp\left(-1/4t  \right) \\
&={1\over 4\pi}\sum_{m\in \Z} {1\over (z_{12}-m)^2}
\end{align*}

$I_2$ is also absolutely convergent. To see this, let
\begin{align*}
F(y)&=\left({\bar z_{12}-(y+n\bar \tau)\over 4t}\right)^2 \exp\left(-|z_{12}-(y+n\tau)|^2/4t  \right)\\
&={1\over (z_{12}-y-n\tau)^2} G(u), \ \ \ \ \ \ \   u=t/|z_{12}-(y+n\tau)|^2
\end{align*}
where $G(u)={1\over (4u)^2}\exp\left(-1/4u\right)$ which is a smooth and bounded function on $[0,\infty)$. Since
\begin{align*}
       {dF(y)\over dy}={2\over (z_{12}-y-n\tau)^3} G(u)+\left({1\over (z_{12}-y-n\tau)^3}+{1\over (z_{12}-y-n\tau)^2(\bar z_{12}-y-n\bar \tau)} \right)u G^\prime(u)
\end{align*}
We can write $I_2$ as
\begin{align*}
   I_2=\int_\epsilon^L {dt\over 4\pi t} \sum_{n\neq 0}\sum_{m\in Z}\int_m^{m+1}dy\bracket{F(y)-F(m)}
\end{align*}
which is of the order ${1\over \abs{m+n\tau}^3}$ as $m,n\to \infty$. Therefore similar to the calculation for $I_1$, we find
\begin{align*}
   \lim_{\substack{\epsilon\to 0\\ L\to \infty}} I_2&={1\over 4\pi} \sum_{n\neq 0}\sum_{m\in \Z}\left({1\over (z_{12}-m-n\tau)^2 }-\int_m^{m+1} dy{1\over (z_{12}-y-n\tau)^2 }   \right)\\
  &={1\over 4\pi} \sum_{n\neq 0}\sum_{m\in \Z}{1\over (z_{12}-m-n\tau)^2 }
\end{align*}

To evaluate $I_3$, notice that
\begin{align*}
&\int_{-\infty}^\infty dy\left({\bar z_{12}-(y+n\bar \tau)\over 4t}\right)^2 \exp\left(-|z_{12}-(y+n\tau)|^2/4t  \right)\\
&=\int_{-\infty}^\infty dy {y^2-(\im z_{12}-n\im \tau)^2\over (4t)^2}\exp\left(-y^2/4t-(\im z_{12}-n\im \tau)^2/4t\right)\\
&=-{\sqrt{\pi}((\im z_{12}-n\im\tau)^2/t-2)\over 8t^{1/2}  }\exp\left(-(\im z_{12}-n\im\tau)^2/4t\right)\\
&=t{d\over dt}\left(-{\pi\over (4\pi t)^{1/2}}\exp\left(-(\im z_{12}-n\im\tau)^2/4t\right) \right)
\end{align*}
Therefore
\begin{align*}
 \lim_{\substack{\epsilon\to 0\\ L\to \infty}}I_3&=-\lim_{\substack{\epsilon\to 0\\ L\to \infty}}\left.{1\over 4} \sum_{n\neq 0}\left({1\over (4\pi t)^{1/2}}\exp\left(-(\im z_{12}-n\im\tau)^2/4t\right) \right)\right|^{L}_\epsilon\\
&=-\lim_{\substack{\epsilon\to 0\\ L\to \infty}}\left.{1\over 4 \im \tau} \sum_{n\neq 0}\left({1\over (4\pi t)^{1/2}}\exp\left(-(a-n)^2/4t\right) \right)\right|^{L}_\epsilon, \ \ \ a=\im z_{12}/\im \tau, 0\leq a<1
\end{align*}

Obviously,
$$
  \lim_{\epsilon\to 0}\sum_{n\neq 0}\left({1\over (4\pi \epsilon)^{1/2}}\exp\left(-(a-n)^2/4\epsilon\right) \right)=0
$$

The Poisson summation formula gives
$$
    \sum_{n\in \Z}\left({1\over (4\pi L)^{1/2}}\exp\left(-(a-n)^2/4 L\right) \right)=\sum_{m\in \Z} \exp\left(-4\pi^2m^2 L+2\pi i m a\right)
$$
hence
\begin{align*}
   &\lim_{L\to \infty}\sum_{n\in \Z}\left({1\over (4\pi L)^{1/2}}\exp\left(-(a-n)^2/4 L\right) \right)\\
&=\lim_{L\to \infty}\sum_{m\in \Z} \exp\left(-4\pi^2m^2 L+2\pi i m a\right)=1
\end{align*}

Adding the three terms together, we find
\begin{align*}
\lim_{\substack{\epsilon\to 0\\ L\to \infty}}\left( \Ptau_{\epsilon,L}(z_1,\bar z_1;z_2,\bar z_2)\right)&={1\over 4\pi}\sum_{n\in \Z}\sum_{m\in \Z}{1\over (z_{12}-m-n\tau)^2}-{1\over 4\im \tau}\\
&={1\over 4\pi}\wp(z_{12};\tau)+{1\over 12\pi}E_2(\tau)-{1\over 4\im \tau}   \\
&={1\over 4\pi}\wp(z_{12};\tau)+{1\over 12\pi}E_2^*(\tau;\bar\tau)
\end{align*}

\end{proof}

\begin{proof}[Proof of Lemma \ref{self-loop}] From the proof of Lemma \ref{bcov propagator}, it's easy to see that
\begin{align*}
   \lim_{\substack{\epsilon\to 0\\L\to\infty}}\lim_{z_1\to z_2}I_1(z_1,z_2)&={1\over 4\pi}\sum_{m\neq 0}{1\over m^2}\\
   \lim_{\substack{\epsilon\to 0\\L\to\infty}}\lim_{z_1\to z_2}I_2(z_1,z_2)&={1\over 4\pi}\sum_{n\geq 0}\sum_{m\in Z}{1\over (m+n\tau)^2}\\
   \lim_{\substack{\epsilon\to 0\\L\to\infty}}\lim_{z_1\to z_2}I_3(z_1,z_2)&=-{1\over 4\pi \im \tau}
\end{align*}
It follows that
$$
 \lim_{\substack{\epsilon\to 0\\L\to\infty}}\lim_{z_1\to z_2}I_2(z_1,z_2) \Ptau_{\epsilon,L}(z_1,z_2)={1\over 12\pi }E_2^*(\tau,\bar\tau)
$$
This proves the case for $n=0$. For $n>0$,
\begin{align*}
 & \pa_{z_1}^n\Ptau_{\epsilon,L}(z_1,z_2)\\=&\int_\epsilon^L {dt\over 4\pi t} \sum_{m,n\in \Z}\bracket{{\bar z_{12}-\bracket{m+n\bar \tau}\over 4t}}^{n+2} \exp\bracket{-\abs{z_{12}-\bracket{m+n\tau}}^2/4t}\\
   =&\int_\epsilon^L{dt\over 4\pi t} \sum_{m,n\in \Z}{1\over \bracket{z_{12}-(m+n\tau)}^{n+2}}{\abs{\bar z_{12}-\bracket{m+n\bar \tau}}^{2}\over 4t}^{n+2} \exp\bracket{-\abs{z_{12}-\bracket{m+n\tau}}^2/4t}
\end{align*}
which is in fact absolutely convergent. Therefore in this case
\begin{align*}
\lim_{\substack{\epsilon\to 0\\L\to\infty}}\lim_{z_1\to z_2}\pa_{z_1}^n\Ptau_{\epsilon,L}(z_1,z_2)&=\sum_{\substack{m,n\in \Z\\ (m,n)\neq (0,0)}}{1\over (m+n\tau)^{n+2}}\int_0^\infty {dt\over 4\pi t}{1\over t^{n+2}}e^{-1/t}\\
&={(n+1)!\over 4\pi}\sum_{\substack{m,n\in \Z\\ (m,n)\neq (0,0)}}{1\over (m+n\tau)^{n+2}}\\
&=\begin{cases}{(n+1)!\zeta(n+2)\over 2\pi}E_{n+2} & \text{if $n$ is even}\\ 0 & \text{if $n$ is odd} \end{cases}
\end{align*}
\end{proof}

\section{Graph integrals on $\C$}\label{appendix-finiteness-lemma}
In this appendix, we will prove some results for graph integrals on $\C$.

Let $z$ be the linear holomorphic coordinate on $\C$, $\Box=-4{\pa\over \pa z}{\pa\over \pa \bar z}$ be the standard Laplacian operator. The following notations will be used throughout this section
$$
          H_{\epsilon}^{L}(z,\bar z)=\int_\epsilon^L {dt\over 4\pi t}e^{-|z|^2/4t}
$$

Let $(\Gamma, n)$ be a decorated graph as in section \ref{section-graph-integral}. We will assume that $\Gamma$ is connected without self-loops. We consider the following graph integral on $\C$
$$
   W_{(\Gamma,n)}(H_\epsilon^L,\Phi)\equiv\prod_{v\in V(\Gamma)} \int_{\C}  d^2z_v
    \left(\prod_{e\in E(\Gamma)}\pa^{n_e}_{z_e}H_\epsilon^L(z_e,\bar z_e) \right)
   \Phi, \ \ \mbox{where}\ z_e=z_{h(e)}-z_{t(e)}
$$
here $\Phi$ is a smooth function on $\mathbb C^{|V(\Gamma)|}$ with compact support. In the above integral, we view $H_\epsilon^L(z_e,\bar z_e)$ as propagators associated to the edge $e\in E$, and we have only holomorphic derivatives on the propagators.

\begin{prop}\label{finiteness lem}
 The following limit  exists for the above graph integral
\begin{eqnarray*}
   \lim_{\epsilon\to 0} W_{(\Gamma,n)}(H_\epsilon^L,\Phi)
\end{eqnarray*}
\end{prop}

\begin{proof} Let $V=|V(\Gamma)|$ be the number of vertices and $E=|E(\Gamma)|$ be the number of edges. We index the vertices by
$$
      v: \{1,2,\cdots, V\} \to V(\Gamma), \ \ \ V=|V(\Gamma)|
$$
and write $z_{i}$ for $z_{v(i)}$ if there's no confusion. We specify the last vertex by $v_\bullet$
$$
    v(V)=v_\bullet
$$

Define the incidence matrix $\{\rho_{v,e}\}_{v\in V(\Gamma), e\in E(\Gamma)}$ by
$$
     \rho_{v,e}=\begin{cases}  1& h(e)=v \\ -1 & t(e)=v\\ 0 & \mbox{otherwise} \end{cases}
$$
and define the $(V-1)\times (V-1)$ matrix $M_\Gamma(t)$ as in \cite{QFT} \S6-2-3 by
\begin{eqnarray}\label{graph matrix}
   M_\Gamma(t)_{i,j}=\sum_{e\in E(G)}\rho_{v(i),e}{1\over t_e}\rho_{v(j),e}, \ \ \ \ \ \ 1\leq i,j\leq V-1
\end{eqnarray}
where $t_e$ is a variable introduced for each edge coming from the propagator. Consider the following linear change of variables
$$
    \begin{cases}
         z_i=y_i+y_V& 1\leq i\leq  V-1\\
         z_V=y_V
    \end{cases}
$$

The graph integral can be written as
\begin{eqnarray*}
    &&W_{(\Gamma,n)}(H_\epsilon^L,\Phi)\\&=&\int_\C d^2y_V \int_{\C^{V-1}}\prod_{i=1}^{V-1}d^2y_i
    \int_{[\epsilon, L]^{E}}\prod_{e\in E(\Gamma)} {dt_e\over 4\pi t_e}  \prod_{e\in E(\Gamma)}\left({\sum\limits_{i=1}^{V-1}\rho_{v(i),e}\bar y_i\over 4 t_e  } \right)^{n_e}
\\&&\exp\left(-{1\over 4}\sum_{i,j=1}^{V-1}M_\Gamma(t)_{i,j}y_{i}\bar y_{j}  \right)\Phi
\end{eqnarray*}
Using integration by parts, we get
\begin{eqnarray*}
W_{(\Gamma,n)}(H_\epsilon^L,\Phi)&=&\int_\C d^2y_V \int_{\C^{V-1}}\prod_{i=1}^{V-1}d^2y_i
  \int_{[\epsilon, L]^{E}} \prod_{e\in E(\Gamma)} {dt_e\over 4\pi t_e}
\exp\left(-{1\over 4}\sum_{i,j=1}^{V-1}M_\Gamma(t)_{i,j}y_{i}\bar y_{j} \right) \\
&&\prod_{e\in E(\Gamma)}\left(\sum\limits_{j=1}^{V-1}{\sum\limits_{i=1}^{V-1}\rho_{v(i),e}M_\Gamma^{-1}(t)_{i,j}\over t_e}  {\pa\over \pa y_{j}} \right)^{n_e} \Phi
\end{eqnarray*}

By Lemma \ref{bound} below, we see that
\begin{eqnarray*}
\left|\prod_{e\in E(\Gamma)}\left(\sum\limits_{j=1}^{V-1}{\sum\limits_{i=1}^{V-1}\rho_{v(i),e}M_\Gamma^{-1}(t)_{i,j}\over t_e}  {\pa\over \pa y_{j}} \right)^{n_e} \Phi\right|\leq
C \left| \tilde \Phi \right|
\end{eqnarray*}
where $C$ is a constant which doesn't depend on $\{t_e\}$ and $\{y_i\}$, and $\tilde \Phi$ is some smooth function with compact support. To prove that  $\lim\limits_{\epsilon\to 0}W_{\Gamma,\{n_e\}}(H_\epsilon^L,\Phi)$ exists, we only need to show that
\begin{eqnarray*}
   &&\lim_{\epsilon\to 0} \int_{\C^{V-1}}\prod_{i=1}^{V-1}d^2y_i
   \int_{[\epsilon, L]^{E}} \prod_{e\in E(\Gamma)} {dt_e\over 4\pi t_e}
\exp\left(-{1\over 4}\sum_{i,j=1}^{V-1}M_\Gamma(t)_{i,j}y_{i}\bar y_{j} \right)\\
&& =\lim_{\epsilon\to 0}\int_{[\epsilon, L]^{E}} \prod_{e\in E(\Gamma)} {dt_e\over 4\pi t_e}{1\over \det M_\Gamma(t)}
\end{eqnarray*}
exists. By Lemma \ref{determinant} below, we have
$$
\lim_{\epsilon\to 0}\int_{[\epsilon, L]^{E}} \prod_{e\in E(\Gamma)} {dt_e\over 4\pi t_e}{1\over \det M_\Gamma(t)}
=\lim_{\epsilon\to 0}\int_{[\epsilon, L]^{E}} \prod_{e\in E(\Gamma)} {dt_e\over 4\pi}{1\over \sum\limits_{T\in \mathrm{Tree}(\Gamma)}\prod\limits_{e
\notin T}{t_e}}
$$
where $\mathrm{Tree}(\Gamma)$ is the set of spanning trees of $\Gamma$. Let $v(1), v(2)$ be two vertices of $\Gamma$, $\{e_1,\cdots, e_k\}$ be the set of edges that connects $v(1), v(2)$. Let $\bar\Gamma$ be the graph obtained from $\Gamma$ by collapsing $v(1)$ and $v(2)$ and all the edges $e_1,\cdots, e_k$ into one single vertex. Then $\bar\Gamma$ is also a connected graph without self-loops, with $E(\bar\Gamma)=E(\Gamma)\backslash \{e_1,\cdots, e_k\}$.  Obviously, for non-negative $t_e$'s,
\begin{eqnarray*}
  {\sum\limits_{T\in \mathrm{Tree}(\Gamma)}\prod\limits_{e
\notin T}{t_e}}\geq  \left(\sum\limits_{i=1}^k {t_{e_1}\cdots \hat t_{e_i} \cdots t_{e_k}} \right){\sum\limits_{T\in {Tree}(\bar\Gamma)}\prod\limits_{e
\notin T}{t_e}}\
\end{eqnarray*}
Therefore
\begin{eqnarray*}
\prod_{e\in E(\Gamma)} \int_\epsilon^L{dt_e\over 4\pi}{1\over \sum\limits_{T\in \mathrm{Tree}(\Gamma)}\prod\limits_{e
\notin T}{t_e}}  &\leq&
\prod\limits_{i=1}^k\int_\epsilon^L{dt_i\over 4\pi}{1\over \sum\limits_{i=1}^k {t_1\cdots \hat t_i \cdots t_k} }\prod_{e\in E(\bar\Gamma)} \int_\epsilon^L{dt_e\over 4\pi}{1\over \sum\limits_{T\in {Tree}(\bar\Gamma)}\prod\limits_{e
\notin T}{t_e}}\\
&\leq&\prod\limits_{i=1}^k\int_\epsilon^L{dt_i\over 4\pi}{k\over \prod\limits_{i=1}^k t_i^{{k-1\over k}}}\prod_{e\in E(\bar\Gamma)} \int_\epsilon^L{dt_e\over 4\pi}{1\over \sum\limits_{T\in {Tree}(\bar\Gamma)}\prod\limits_{e
\notin T}{t_e}}\\
&\leq& C(L)\prod_{e\in E(\bar\Gamma)} \int_\epsilon^L{dt_e\over 4\pi}{1\over \sum\limits_{T\in {Tree}(\bar\Gamma)}\prod\limits_{e
\notin T}{t_e}}
\end{eqnarray*}
where $C(L)$ is a constant that depends only on $L$. By successive collapsing of vertices, we see that $\lim\limits_{\epsilon\to 0}\int_{[\epsilon, L]^{E}} \prod_{e\in E(\Gamma)} {dt_e\over 4\pi t_e}{1\over \det M_\Gamma(t)} $ exists. This proves the lemma.
\end{proof}

\begin{dfn}
A tree $T\subset \Gamma$ is said to be a spanning tree for the connected graph $\Gamma$ if every vertex of $\Gamma$ lies in $T$.
\end{dfn}

\begin{lem}\label{determinant}The determinant of the $(V-1)\times (V-1)$ matrix $\{M_\Gamma(t)_{i,j}\}_{1\leq i,j\leq V-1}$ defined by equation (\ref{graph matrix}) is given by
\begin{eqnarray}
    \det\ M_\Gamma(t)=\sum_{T\in \mathrm{Tree}(\Gamma)}\prod_{e\in T}{1\over t_e}
\end{eqnarray}
where $\mathrm{Tree}(\Gamma)$ is the set of spanning trees of the graph $\Gamma$.
\end{lem}
\begin{proof} See for example \cite{QFT} \S 6-2-3.
\end{proof}

\begin{dfn}
Given a connected graph $\Gamma$ and two disjoint subsets of vertices $V_1, V_2\subset V(\Gamma)$, $V_1\cap V_2=\emptyset$,  we define $\mathrm{Cut}(\Gamma; V_1, V_2)$ to be the set of subsets $C\subset E(\Gamma)$ satisfying the following property
\begin{enumerate}
\item The removing of the edges in $C$ from $\Gamma$ divides $\Gamma$ into exactly two connected trees, which we denoted by $\Gamma_1(C), \Gamma_2(C)$, such that
$V_1\subset V(\Gamma_1(C)), V_2\subset V(\Gamma_2(C))$.
\item $C$ doesn't contain any proper subset satisfying property $1$.
\end{enumerate}
\end{dfn}
It's easy to see that each cut $C\in \mathrm{Cut}(\Gamma; V_1, V_2)$ is obtained by adding one more edge to some $\{e\in E(\Gamma)|e\notin T\}$ where $T$ is some spanning tree of $\Gamma$. Then we have the following result; see \cite{QFT} \S 6-2-3.

\begin{lem}\label{cut}The inverse of the matrix $M_\Gamma(t)$ is given by
\begin{eqnarray*}
   M_\Gamma^{-1}(t)_{i,j}={1\over \mc P_\Gamma(t)}\sum_{C\in \mathrm{Cut}(\Gamma;\{v(i),v(j)\},\{v_\bullet\})} \prod_{e\in C}t_e
\end{eqnarray*}
where
\begin{eqnarray*}
   \mc P_\Gamma(t)=\sum_{T\in \mathrm{Tree}(\Gamma)} \prod_{e\not\in T}t_e=\det\ M_\Gamma(t)\prod_{e\in E(\Gamma)}t_e
\end{eqnarray*}
\end{lem}
\begin{proof} Let
$$
A_{i,j}={1\over \mc P_\Gamma(t)}\sum_{C\in \mathrm{Cut}(\Gamma;\{v(i),v(j)\},\{v_\bullet\})} \prod_{e\in C}t_e
$$

For $1\leq i\leq V-1$, consider the summation
\begin{eqnarray*}
       \mc P_\Gamma(t)\sum_{j=1}^{V-1}A_{i,j}M_\Gamma(t)_{j,i}&=&\sum_{j=1}^{V-1} M_\Gamma(t)_{j,i}\sum_{{C\in \mathrm{Cut}(\Gamma;\{v(i),v(j)\},\{v_\bullet\})}} \prod_{e\in C}t_e\\
       &=&\sum_{\substack{C\in \mathrm{Cut}(\Gamma;\{v(i)\},\{v_\bullet\})\\ v(i)\in V(\Gamma_1(C)), v_\bullet\in V(\Gamma_2(C)) }}\prod_{e\in C}t_e\sum_{e^\prime\in E(G)}\sum_{\substack{1\leq j\leq V-1\\v(j)\in \Gamma_1(C)}}\rho_{v(i),e^\prime}{1\over t_{e^\prime}}\rho_{v(j),e^\prime}\\
       &=&\sum_{\substack{C\in \mathrm{Cut}(\Gamma;\{v(i)\},\{v_\bullet\})\\ v(i)\in V(\Gamma_1(C)), v_\bullet\in V(\Gamma_2(C)) }}\prod_{e\in C}t_e\sum_{\substack{e^\prime\in E(G)\\ l(e)=v(i), r(e)\in V(\Gamma_2)\\
       \text{or}\ r(e)=v(i), l(e)\in V(\Gamma_2)}}{1\over t_{e^\prime}}\\
       &=&\sum_{T\in \mathrm{Tree}(\Gamma)} \prod_{e\not\in T}t_e
      \end{eqnarray*}
where in the last step, we use the fact that given $v\neq v_\bullet$ and a spanning tree $T$ of $\Gamma$, there's a unique way to remove one edge in $T$, which is attached to $v$,  to make a cut that separates $v$ and $v_\bullet$. Therefore
$$
   \sum_{j=1}^{V-1} A_{i,j}M_\Gamma(t)_{j,i}=1, \ \ 1\leq i\leq V-1
$$

Similar combinatorial interpretation leads to
$$
   \sum_{k=1}^{V-1} A_{i,k}M_\Gamma(t)_{k,j}=0, \ \ 1\leq i, j\leq V_1, i\neq j
$$
We leave the details to the reader. It follows that $A_{i,j}$ is the inverse matrix of $M_\Gamma(t)_{i,j}$. \end{proof}

\begin{lem}\label{bound}The following sum is bounded
$$
      \left| \sum\limits_{i=1}^{V-1}\rho_{v(i),e} M_\Gamma^{-1}(t)_{i,j}\over t_e \right|\leq 2, \ \ \forall e\in E(G), 1\leq j\leq V-1
$$
\end{lem}
\begin{proof}
\begin{eqnarray*}
&&\sum\limits_{i=1}^{V-1}{\rho_{v(i),e}\over t_e }M_\Gamma^{-1}(t)_{i,j}\\&=&
{1\over \mc P_\Gamma(t)}\sum\limits_{\substack{C\in \mathrm{Cut}(\Gamma; \{v(j)\}, \{v\bullet\})\\ v(j)\in V(\Gamma_1(C)), v_\bullet\in V(\Gamma_2(C))}}\prod_{e^\prime\in C}t_{e^\prime} \sum_{\substack{1\leq i\leq V-1\\ v(i)\in \Gamma_1(C)}}{\rho_{v(i),e}\over t_e}\\
&=&{1\over \mc P_\Gamma(t)}\sum\limits_{\substack{C\in \mathrm{Cut}(\Gamma; \{v(j), l(e)\}, \{v_\bullet,r(e)\})}}{\prod_{e^\prime\in C}t_{e^\prime} \over t_e}
-{1\over \mc P_\Gamma(t)}\sum\limits_{\substack{C\in \mathrm{Cut}(\Gamma; \{v(j), r(e)\}, \{v_\bullet,l(e)\})}}{\prod_{e^\prime\in C}t_{e^\prime} \over t_e}
\end{eqnarray*}

Since each cut in the above summation is obtained from removing the edge $e$ from a spanning tree containing $e$, the lemma follows from fact that  $\mc P_\Gamma(t)=\sum_{T\in \mathrm{Tree}(\Gamma)} \prod_{e\not\in T}t_e$
represents the sum of the contributions from all such spanning trees.
\end{proof}

\vjump{3}

Next, we consider another type of graph integral which appears in the proof of Proposition \ref{antiholomorphic-dependence}. Let
\begin{align*}
    P_\epsilon^L(z,\bar z)=\int_\epsilon^L {dt\over 4\pi t}\bracket{\bar z\over 4t}^2 e^{-|z|^2/4t}
\end{align*}
and
\begin{align*}
    U_\epsilon(z,\bar z)={1\over 4\pi \epsilon}\bracket{\bar z\over 4\epsilon}e^{-|z|^2/4\epsilon}
\end{align*}

Let $(\Gamma, n)$ be a connected decorated graph without self-loops, $V(\Gamma)$ be the set of vertices, $E(\Gamma)$ be the set of edges, $V=|V(\Gamma)|, E=|E(\Gamma)|$. We index the set of vertices as in Proposition \ref{finiteness lem} by
$$
    v:\{1,2,\cdots, V\}\to V(\Gamma)
$$
and index the set of edges by
$$
   e: \{0,1,2,\cdots, E-1\}\to E(\Gamma)
$$
such that $e(0),e(1),\cdots,e(k)\in E(\Gamma)$ are all the edges connecting $v(1),v(V)$. We consider the following Feynman graph integral by putting $U_\epsilon$ on $e(0)$, putting $P_\epsilon^L$ to all other edges, and putting a smooth function $\Phi$ on $\C^{|V(\Gamma)|}$ with compact support  for the vertices. We would like to compute the following limit of the graph integral
$$
    \lim_{\epsilon\to 0}\prod_{i=1}^{V}\int d^2z_i
\pa^{n_0}_{z_{e(0)}}U_\epsilon(z_{e(0)}, \bar z_{e(0)})\left(\prod\limits_{i=1}^{E-1}\pa^{n_i}_{z_{e(i)}}P_\epsilon^L(z_{e(i)}, \bar z_{e(i)})\right)\Phi
$$
where we use the notation that
$$
     z_e\equiv z_{i}-z_{j}, \ \ \mbox{if}\ h(e)=v(i), t(e)=v(j)
$$
\begin{prop}\label{deformed-bracket}
The above limit exists and we have the identity
\begin{align*}
\begin{split}
& \lim_{\epsilon\to 0}\prod_{i=1}^{V}\int d^2z_i
\pa^{n_0}_{z_{e(0)}}U_\epsilon(z_{e(0)}, \bar z_{e(0)})\left(\prod\limits_{i=1}^{E-1}\pa^{n_i}_{z_{e(i)}}P_\epsilon^L(z_{e(i)}, \bar z_{e(i)})\right)\Phi\\
=&\lim_{\epsilon\to 0}{A(n_0;n_1,\cdots,n_k)\over (4\pi)^k}\prod_{i=2}^{V}\int d^2z_i \left.\pa_{z_1}^{n_0+1+\sum\limits_{i=1}^k(n_i+2)}\left( \left(\prod\limits_{i=k+1}^{E-1}\pa^{n_i}P_\epsilon^L(z_{e(i)}, \bar z_{e(i)})\right)\Phi \right)\right|_{z_1=z_V}
\end{split}
\end{align*}
where the constant $A(n_0,n_1,\cdots,n_k)$ is a rational number given by
$$
     A(n_0;n_1,\cdots,n_k)=\int_0^1\cdots \int_0^1 \prod\limits_{i=1}^{k}{du_i} {\prod\limits_{i=1}^k u_i^{n_i+1}\over \left(1+\sum\limits_{i=1}^k u_i\right)^{\sum\limits_{j=0}^k (n_j+2)}}
$$
\end{prop}
\begin{proof}
\begin{align*}
&\prod_{i=1}^{V}\int d^2z_i \pa^{n_0}U_\epsilon(z_{e(0)})\left(\prod\limits_{i=1}^{E-1}\pa^{n_i}H_\epsilon^L(z_{e(i)})\right)\Phi\\
=&  \prod_{i=1}^{V}\int d^2z_i \prod_{i=1}^{E-1} \int_\epsilon^L dt_{e(i)} \left({1\over 4\pi \epsilon}\left({\bar z_{e(0)}\over 4\epsilon}\right)^{n_0+1}\right)\left(\prod_{i=1}^{E-1}{1\over 4\pi t_{e_(i)}}\left( {\bar z_{e(i)}\over 4t_{e(i)}} \right)^{n_i+2}  \right)  e^{-\left({|z_{e(0)}|^2\over 4\epsilon}+\sum\limits_{j=1}^{E-1}{|z_{e(j)}|^2\over 4t_{e(i)}}\right)}\Phi
\end{align*}

We will use the same notations as in the proof of Proposition \ref{finiteness lem}. The incidence matrix $\{\rho_{v,e}\}_{v\in V(G), e\in E(G)}$ is defined by
$$
     \rho_{v,e}=\begin{cases}  1& h(e)=v \\ -1 & t(e)=v\\ 0 & \mbox{otherwise} \end{cases}
$$
Without loss of generality, we assume that the orientation of $e(0)$ is such that
$$
   \rho_{v(1),e(0)}=1, \rho_{v(V),e(0)}=-1
$$
The $(V-1)\times (V-1)$ matrix $M_\Gamma(t)$ is defined by
$$
   M_\Gamma(t)_{i,j}=\sum_{l=0}^{E-1}\rho_{v(i),e(l)}{1\over t_{e(l)}}\rho_{v(j),e(l)}, \ \ \ \ \ \ 1\leq i,j\leq V-1
$$
where we use the convention  that $t_{e(0)}=\epsilon$. Under the following linear change of variables
$$
    \begin{cases}
         z_{i}=y_{i}+y_{V}& 1\leq i\leq  V-1\\
         z_{V}=y_{V}
    \end{cases}
$$
and use integration by parts

\begin{align*}
   &\prod_{i=1}^{V}\int d^2z_i \prod_{i=1}^{E-1} \int_\epsilon^L dt_{e(i)} \left({1\over 4\pi \epsilon}\left({\bar z_{e(0)}\over 4\epsilon}\right)^{n_0+1}\right)\left(\prod_{i=1}^{E-1}{1\over 4\pi t_{e_(i)}}\left( {\bar z_{e(i)}\over 4t_{e(i)}} \right)^{n_i+2}  \right)  e^{-\left({|z_{e(0)}|^2\over 4\epsilon}+\sum\limits_{j=1}^{E-1}{|z_{e(j)}|^2\over 4t_{e(i)}}\right)}\Phi\\
   \\=&  \int d^2y_V \prod_{i=1}^{V-1}\int d^2y_i \prod_{i=1}^{E-1}\int_\epsilon^L {dt_{e(i)}\over 4\pi t_{e(i)}} \exp\left(-{1\over 4}\sum_{i,j=1}^{V-1}y_{i}M_{\Gamma}(t)_{i,j}\bar y_{j} \right)\\
&{1\over 4\pi \epsilon} \left(\sum\limits_{j=1}^{V-1}{\sum\limits_{i=1}^{V-1}\rho_{v(i),e(0)}M_\Gamma^{-1}(t)_{i,j}\over \epsilon } {\pa\over \pa y_j}\right)^{n_0+1}
\prod\limits_{\alpha=1}^{E-1}\left(\sum\limits_{j=1}^{V-1}{\sum\limits_{i=1}^{V-1}\rho_{v(i),e(\alpha)}M_\Gamma^{-1}(t)_{i,j}\over t_{e(\alpha)} } {\pa\over \pa y_j}\right)^{n_\alpha+2}\Phi
\end{align*}
Note that for $0\leq \alpha\leq k$ and $1\leq i\leq V-1$, $\rho_{v(i),e(\alpha)}$ is nonzero only for $\rho_{v(1),e(\alpha)}=1$. Consider the change of variables
\begin{eqnarray*}
\begin{array}{lc}
      t_{e(i)}\to \epsilon t_{e(i)} & 1\leq i\leq k\\
      t_{e(i)}\to t_{e(i)} & k+1\leq i\leq E-1
\end{array}
\end{eqnarray*}
we get
\begin{align*}
 &\int d^2y_V \prod_{i=1}^{V-1}\int d^2y_i \prod_{i=1}^{k}\int_1^{L/\epsilon} {dt_{e(i)}\over 4\pi t_{e(i)}}
 \prod_{i=k+1}^{E-1}\int_\epsilon^{L} {dt_{e(i)}\over 4\pi t_{e(i)}}  \exp\left(-{1\over 4}\sum_{i,j=1}^{V-1}y_{i}\bar y_{j}M_{\Gamma}(\tilde t)_{i,j} \right)\\
&{1\over 4\pi \epsilon} \left(\sum\limits_{j=1}^{V-1}{\sum\limits_{i=1}^{V-1}\rho_{v(i),e(0)}M_\Gamma^{-1}(\tilde t)_{i,j}\over \epsilon } {\pa\over \pa y_j}\right)^{n_0+1}\prod\limits_{\alpha=1}^{E-1}\left(\sum\limits_{j=1}^{V-1}{\sum\limits_{i=1}^{V-1}\rho_{v(i),e(\alpha)}M_\Gamma^{-1}(\tilde t)_{i,j}\over \tilde t_{e(\alpha)} } {\pa\over \pa y_j}\right)^{n_\alpha+2}\Phi\\
=&\prod_{i=1}^{k}\int_1^{L/\epsilon} {dt_{e(i)}\over 4\pi}
 \prod_{i=k+1}^{E-1}\int_\epsilon^{L} {dt_{e(i)}\over 4\pi }F(t;\epsilon)
\end{align*}
where $\tilde t$'s are define by
\begin{eqnarray*}
        \begin{array}{lc}
              \tilde t_{e(0)}=\epsilon&\\
              \tilde t_{e(i)}=\epsilon t_{e(i)} & \text{if}\ 1\leq i\leq k\\
              \tilde t_{e(i)}=t_{e(i)} & \text{if}\ k+1\leq i\leq E-1
        \end{array}
\end{eqnarray*}
and
\begin{align*}
   F(t;\epsilon)=& \prod_{i=1}^{V}\int d^2y_i {1\over \prod\limits_{i=1}^{E-1}t_{e(i)}}\exp\left(-{1\over 4}\sum_{i,j=1}^{V-1}y_{i}\bar y_{j}M_{\Gamma}(\tilde t)_{i,j} \right)\\
&{1\over 4\pi \epsilon} \left(\sum\limits_{j=1}^{V-1}{\sum\limits_{i=1}^{V-1}\rho_{v(i),e(0)}M_\Gamma^{-1}(\tilde t)_{i,j}\over \epsilon } {\pa\over \pa y_j}\right)^{n_0+1}\prod\limits_{\alpha=1}^{E-1}\left(\sum\limits_{j=1}^{V-1}{\sum\limits_{i=1}^{V-1}\rho_{v(i),e(\alpha)}M_\Gamma^{-1}(\tilde t)_{i,j}\over \tilde t_{e(\alpha)} } {\pa\over \pa y_j}\right)^{n_\alpha+2}\Phi
\end{align*}

We first show that $\lim\limits_{\epsilon\to 0}F(t;\epsilon)$ exists. Using integration by parts,
\begin{align*}
F(t;\epsilon)=&\prod_{i=1}^{V}\int d^2y_i {1\over \prod\limits_{i=1}^{E-1}t_{e(i)}}
\exp\left(-{|y_1|^2\over 4\epsilon}\left(1+\sum\limits_{\alpha=1}^k {1\over t_{e(\alpha)}}\right) \right)
{1\over 4\pi \epsilon} \left({\bar y_1\over 4\epsilon}\right)^{n_0+1}
\prod\limits_{\alpha=1}^{k}\left({\bar y_1\over 4\epsilon t_{e(\alpha)}}\right)^{n_\alpha+2}
\\
&
\exp\left(-{1\over 4}\sum_{i,j=1}^{V-1}y_{i}\bar y_{j}\sum\limits_{\beta=k+1}^{E-1}{\rho_{v(i),e({\beta})}\rho_{v(j),e(\beta)}\over t_{e(\beta)}} \right)
\prod\limits_{\beta=k+1}^{E-1}\left({\sum\limits_{i=1}^{V-1}\rho_{v(i),e(\beta)}\bar y_i\over 4t_{e(\beta)}}\right)^{n_\alpha+2}\Phi\\
=&\prod_{i=1}^{V}\int d^2y_i {1\over \prod\limits_{i=1}^{E-1}t_{e(i)}}\\
&\exp\left(-{|y_1|^2\over 4\epsilon}\left(1+\sum\limits_{\alpha=1}^k {1\over t_{e(\alpha)}}\right) \right)
{1\over 4\pi \epsilon}\left(\prod\limits_{\alpha=1}^k {1\over t_{e(\alpha)}}\right)^{n_\alpha+2}
{1\over \left(1+\sum\limits_{\alpha=1}^k{1\over t_{e(\alpha)}}\right)^{n_0+1+\sum\limits_{\alpha=1}^k(n_\alpha+2)}}
\\
& \left({\pa\over \pa y_1}\right)^{n_0+1+\sum\limits_{\alpha=1}^k(n_\alpha+2)}
\left(e^{-{1\over 4}\suml_{i,j=1}^{V-1}y_{i}\bar y_{j}\sum\limits_{\beta=k+1}^{E-1}{\rho_{v(i),e({\beta})}\rho_{v(j),e(\beta)}\over t_{e(\beta)}} }
\prod\limits_{\beta=k+1}^{E-1}\left({\sum\limits_{i=1}^{V-1}\rho_{v(i),e(\beta)}\bar y_i\over 4t_{e(\beta)}}\right)^{n_\alpha+2}\Phi\right)\\
\end{align*}
Using the property of the heat kernel under the limit $\epsilon\to 0$, we get
\begin{align*}
\lim\limits_{\epsilon\to 0}F(t;\epsilon)=&\prod_{i=2}^{V}\int d^2y_i{1\over \prod\limits_{i=1}^{E-1}t_{e(i)}}\left(\prod\limits_{\alpha=1}^k {1\over t_{e(\alpha)}}\right)^{n_\alpha+2}
{1\over \left(1+\sum\limits_{\alpha=1}^k{1\over t_{e(\alpha)}}\right)^{\sum\limits_{\alpha=0}^k(n_\alpha+2)}}\\
&\left.\left({\pa\over \pa y_1}\right)^{n_0+1+\sum\limits_{\alpha=1}^k(n_\alpha+2)}
\left(e^{-{1\over 4}\suml_{i,j=1}^{V-1}y_{i}\bar y_{j}\sum\limits_{\beta=k+1}^{E-1}{\rho_{v(i),e({\beta})}\rho_{v(j),e(\beta)}\over t_{e(\beta)}} }
\prod\limits_{\beta=k+1}^{E-1}\left({\sum\limits_{i=1}^{V-1}\rho_{v(i),e(\beta)}\bar y_i\over 4t_{e(\beta)}}\right)^{n_\alpha+2}\Phi\right)\right|_{y_1=0}
\end{align*}

\begin{claim}
$$
\lim\limits_{\epsilon\to 0}\prod_{i=1}^{k}\int_1^{L/\epsilon} {dt_{e(i)}\over 4\pi}
 \prod_{i=k+1}^{E-1}\int_\epsilon^{L} {dt_{e(i)}\over 4\pi }F(t;\epsilon)=
 \prod_{i=1}^{k}\int_1^\infty {dt_{e(i)}\over 4\pi}
 \prod_{i=k+1}^{E-1}\int_0^{L} {dt_{e(i)}\over 4\pi }\lim\limits_{\epsilon\to 0}F(t;\epsilon)
 $$
\end{claim}
Clearly Proposition \ref{deformed-bracket} follows from the claim. \\ \\
\indent To prove the claim, first notice that we have the estimate
$$
  0\leq {M_\Gamma^{-1}(t)_{1,j}\over t_{e(\alpha)} } \leq {1\over t_{e(\alpha)}\left({1\over \epsilon}+\sum\limits_{i=1}^k{1\over t_{e(i)}}\right)}
$$
 for $1\leq \alpha\leq k, 1\leq j\leq V-1$. In fact, by Lemma \ref{cut},
\begin{align*}
 M_\Gamma^{-1}(t)_{1,j}=&{\sum\limits_{C\in \mathrm{Cut}(\Gamma;\{v(1),v(j)\},\{v_V\})} \prod\limits_{e\in C}t_e\over \sum\limits_{T\in \mathrm{Tree}(\Gamma)} \prod\limits_{e\not\in T}t_e}\\
 \leq &{\sum\limits_{C\in \mathrm{Cut}(\Gamma;\{v(1),v(j)\},\{v_V\})} \prod\limits_{e\in C}t_e\over \sum\limits_{\substack{T\in \mathrm{Tree}(\Gamma)\\ e_{i}\in E(T)\ \text{for some}\ 0\leq i\leq k}} \prod\limits_{e\not\in T}t_e}\leq {1\over \left({1\over \epsilon}+\sum\limits_{i=1}^k{1\over t_{e(i)}}\right)}
\end{align*}
For $0\leq \alpha\leq E-1, 1\leq j\leq V-1$, $\sum\limits_{i=1}^{V-1}\rho_{v(i),e(\alpha)}M_\Gamma^{-1}(t)_{i,j}\over t_{e(\alpha)}$ is bounded by a constant by Lemma \ref{bound}. It follows that
\begin{eqnarray*}
 &&  |F(t;\epsilon)|\\&\leq & \prod_{i=1}^{V}\int d^2y_i {1\over \prod\limits_{i=1}^{E-1}t_{e(i)}}\exp\left(-{1\over 4}\sum_{i,j=1}^{V-1}y_{i}\bar y_{j}M_{\Gamma}(\tilde t)_{i,j} \right)\\
&&\left|{1\over 4\pi \epsilon} \left(\sum\limits_{j=1}^{V-1}{\sum\limits_{i=1}^{V-1}\rho_{v(i),e(0)}M_\Gamma^{-1}(\tilde t)_{i,j}\over \epsilon } {\pa\over \pa y_j}\right)^{n_0+1}\prod\limits_{\alpha=1}^{E-1}\left(\sum\limits_{j=1}^{V-1}{\sum\limits_{i=1}^{V-1}\rho_{v(i),e(\alpha)}M_\Gamma^{-1}(\tilde t)_{i,j}\over \tilde t_{e(\alpha)} } {\pa\over \pa y_j}\right)^{n_\alpha+2}\Phi\right|\\
&\leq& \prod_{i=1}^{V}\int d^2y_i {1\over \prod\limits_{i=1}^{E-1}t_{e(i)}}\exp\left(-{1\over 4}\sum_{i,j=1}^{V-1}y_{i}\bar y_{j}M_{\Gamma}(\tilde t)_{i,j} \right){1\over 4\pi \epsilon}\prod\limits_{1\leq \alpha\leq k}\left( {1\over t_{e(\alpha)}\left({1}+\sum\limits_{i=1}^k{1\over t_{e(i)}}\right)}\right)^{n_\alpha+2}\tilde \Phi\\
\end{eqnarray*}
where $\tilde \Phi$ is some non-negative smooth function on $\C^{V}$ with compact support. Integrating over $y_i$'s we get
\begin{align*}
|F(t;\epsilon)|\leq & C {1\over \prod\limits_{i=1}^{E-1}t_{e(i)}}{1\over \epsilon \det M_{\Gamma}(\tilde t)}\prod\limits_{1\leq \alpha\leq k}\left( {1\over t_{e(\alpha)}\left({1}+\sum\limits_{i=1}^k{1\over t_{e(i)}}\right)}\right)^{n_\alpha+2}\\
=&C {\epsilon^k\over  \mc P_{\Gamma}(\epsilon, \epsilon t_{e(1)}, \cdots, \epsilon t_{e(k)}, t_{e_{k+1}}, \cdots, t_{e(E-1)})}\prod\limits_{1\leq \alpha\leq k}\left( {1\over t_{e(\alpha)}\left({1}+\sum\limits_{i=1}^k{1\over t_{e(i)}}\right)}\right)^{n_\alpha+2}\\
\leq&C {1\over \mc P_{\bar\Gamma}(t_{e_{k+1}}, \cdots, t_{e(E-1)})\prod\limits_{\alpha=1}^k t_{e(\alpha)}}\prod\limits_{\alpha=1}^k {1\over t_{e(\alpha)}^{n_\alpha+2}
}\end{align*}
where $C$ is a constant that only depends on $\tilde \Phi$, $\bar\Gamma$ is the graph obtained by collapsing the vertices $v(1), v(V)$ and all $e(0),e(1),\cdots,e(k)$, and $\mc P_\Gamma$ is defined in Lemma \ref{cut}. Here we have used the simple combinatorial fact that
$$
     \mc P_{\Gamma}(\epsilon, \epsilon t_{e(1)}, \cdots, \epsilon t_{e(k)}, t_{e_{k+1}}, \cdots, t_{e(E-1)})\geq \epsilon^k \left(\prod\limits_{\alpha=1}^kt_{e(\alpha)} \right)\left(1+\sum\limits_{\alpha=1}^k {1\over t_{e(\alpha)}}\right)\mc P_{\bar\Gamma}(t_{e_{k+1}}, \cdots, t_{e(E-1)})
$$
Since $\bar\Gamma$ has no self-loops,
$$
\prod_{i=1}^{k}\int_1^\infty {dt_{e(i)}\over 4\pi}\prod\limits_{\alpha=1}^k {1\over t_{e(\alpha)}^{n_\alpha+3}}
 \prod_{i=k+1}^{E-1}\int_0^{L} {dt_{e(i)}\over 4\pi }{1\over \mc P_{\bar\Gamma}(t_{e_{k+1}}, \cdots, t_{e(E-1)})}<\infty
$$
Now the claim follows from dominated convergence theorem.
\end{proof}


\vskip 1em
\noindent \text{S. Li, Mathematics Department, Northwestern University, Evanston, IL 60201. }
\end{document}